\documentclass[letterpaper,11pt,reqno]{amsart}

\usepackage{amsmath}%
\usepackage{amsfonts}%
\usepackage{amssymb}%
\usepackage{diagmac2} 
\usepackage{stmaryrd} 
\usepackage{comment} 

\usepackage{MnSymbol}

\usepackage{scalerel}
\usepackage{verbatim}

\usepackage{enumerate}

\DeclareMathAlphabet{\mathpzc}{OT1}{pzc}{m}{it}

\newtheorem{theo}{Theorem}[section]  
\newtheorem{theorem}[theo]{theorem}
\newtheorem{cor}[theo]{Corollary}
\newtheorem{prop}[theo]{Proposition}

\theoremstyle{definition}
\newtheorem{definition}[theo]{Definition}
\newtheorem{example}[theo]{Example}

\newtheorem{problem}{Problem}

\SetSymbolFont{stmry}{bold}{U}{stmry}{m}{n}

\newcommand{\Alg}[1]{\mathbf{#1}}
\newcommand{\alg}[1]{\mathsf{#1}}
\newcommand{\var}[1]{\mathcal{#1}}

\newcommand{\Qvar}[1]{\mathcal{Q}(#1)}

\newcommand{\qvar}{\mathcal{Q}}
\newcommand{\vvar}{\mathcal{V}}

\newcommand{\Pvar}{\mathcal{P}}

\newcommand{\ruleR}{\mathsf{r}}
\newcommand{\Rules}{\mathsf{R}}

\newcommand{\Lang}{\mathcal{L}}

\newcommand{\CPC}{\mathsf{CPC}}
\newcommand{\IPC}{\mathsf{IPC}}

\newcommand{\Int}{\mathsf{Int}}
\newcommand{\KF}{\mathsf{K4}}
\newcommand{\KP}{\mathsf{KP}}
\newcommand{\ML}{\mathsf{ML}}

\newcommand{\BDn}{\mathsf{BD}_n}

\newcommand{\DS}{\mathsf{S}}
\newcommand{\rules}{\mathsf{R}}
\newcommand{\AllRules}{\mathsf{Rls}}  
\newcommand{\Rule}{\mathsf{r}}

\newcommand{\LogL}{\mathsf{L}}

\newcommand{\For}{\mathsf{For}}
\newcommand{\Con}{\mathcal{C}}
\newcommand{\Vars}{\mathcal{P}}
\newcommand{\Ax}{\mathpzc{Ax}}

\newcommand{\bydef}{:=}

\newcommand{\Frm}{\mathcal{F}}

\newcommand{\HSCpl}{\it HSCpl} 
\newcommand{\SCpl}{\mathcal{SCpl}}

\newcommand{\CDS}{{\DS_c}}
\newcommand{\IDS}{{\DS_i}}
\newcommand{\MP}{\mathsf{mp}}

\newcommand{\FAlg}[2]{\Alg{F}_{#1}(#2)}

\newcommand{\set}[2]{\{ {#1}: {#2}\}}

\newcommand{\ds}[2]{\langle {#1}, {#2}\rangle}
\newcommand{\Ds}{\ds{\Ax}{\Rules}}

\newcommand{\LogDs}[1]{\mathsf{L}(#1)}

\newcommand{\SC}[1]{\widetilde{#1}}
\newcommand{\LogDS}{\LogDs{\DS}}

\newcommand{\SCDS}{\widetilde{\DS}}
\newcommand{\SCR}{\SC{\Rules}(\DS)}

\newcommand{\one}{\mathbf{1}}
\newcommand{\zero}{\mathbf{0}}

\newcommand{\Heyt}{\mathcal{H}}

\newcommand{\LogExt}{\lesssim}
\newcommand{\DedExt}{\lessapprox}

\newcommand{\LogEqv}{\sim}
\newcommand{\DedEqv}{\approx}

\newcommand{\admDS}{\ \raisebox{2pt}{\scaleobj{.7}{|\!\!\!\sim_{\DS}}} \ }

\title[Hereditarily Structurally Complete Deductive Systems]{Hereditarily Structurally Complete Superintuitionistic Deductive Systems}

\author{Alex Citkin}

\address[A.~Citkin]{
Metropolitan Telecommunications, USA}
\email{acitkin@gmail.com}

\begin{document}

 


\begin{abstract} The paper studies hereditarily complete superintuitionistic deductive systems, that is, the deductive system which logic is an extension of the intuitionistic propositional logic. It is proven that for deductive systems a criterion of hereditary structurality - similar to one that exists for logics - does not exists. Nevertheless, it is proven that many standard superintuitionistic logics (including Int) can be defined by a hereditarily structurally complete deductive system.  
\end{abstract}


\maketitle

\section{Introduction}

The notion of structural completeness was introduced by W.~Pogorzelski in \cite{Pogorzelski_Structural_1971}: a (propositional) deductive system\footnote{Pogorzelski is using the term "calculus", but we prefer the term "deductive system", saving the term "calculus" for the deductive systems with finite sets of axiom schemes and rules. In \cite{Olson_Raftery_Alten_2008} a term "formal system" has been used.} $\DS$ is \textit{structurally complete} if every admissible in $\DS$ rule is derivable in $\DS$ (and we denote this by $\DS \in \SCpl$).

First, let us clarify the above definition. Let $\For$ be a set of all (propositional) formulas built in a usual way from an infinite countable set $\Vars$ of (propositional) variables and from a finite set  $\Con$ of connectives. A (structural) \textit{rule} is an ordered pair $\Gamma/B$, where $\Gamma$ is a finite (maybe empty) set of formulas, and $B$ is a formula. \textit{Deductive system} is understood as a pair $\Ds$, where $\Ax$ is a set of formulas and $\Rules$ is a set of rules. And each deductive system $\DS$ defines in a natural way a consequence relation denoted by $\vdash_\DS$ and a logic $\set{A \in \For}{\vdash_\DS A}$ denoted by $\LogDS$. Given a deductive system $\DS \bydef \Ds$, a rule $\ruleR \bydef \Gamma/B$ is \textit{admissible} in $\DS$ if $\LogDS = \LogDs{\ds{\Ax}{\Rules + \ruleR}}$, and $\ruleR$ is \textit{derivable} in $\DS$ if $\Gamma \vdash_\DS B$, that is, if $\vdash_{\ds{\Ax}{\Rules}} \ = \ \vdash_{\ds{\Ax}{\Rules + \ruleR}}$. In other words, a rule $\ruleR$ is admissible in a deductive system $\DS$ if extending $\Rules$ by $\ruleR$ does not change the logic of $\DS$, while $\ruleR$ is derivable in $\DS$ if extending $\Rules$ by $\ruleR$ does not change derivability in $\DS$. 

For instance, let us consider classical deductive system (calculus) $\CDS \bydef \ds{\Ax^c}{\MP}$, where $\Ax^c$ is a set of classical axiom schemes and  $\MP \bydef A, A \to B/B$ is Modus Ponens (see e.g.  \cite{Mendelson_Introduction_2015}[Section 1.4]); and let us consider intuitionistic deductive system $\CDS \bydef \ds{\Ax^i}{\MP}$, where $\Ax^i$ is a set of intuitionistic axiom schemes (see e.g. \cite{Mendelson_Introduction_2015}[Section 1.6]). A consequence relation $\vdash_\CDS$ is structurally complete, while consequence relation $\vdash_\IDS$ is not structurally complete, for the rule $\neg p \to (q \lor r)/(\neg p \to q) \lor (\neg p \to r)$ is admissible in $\vdash_\IDS$ but not derivable in it.

Soon after the notion of structural completeness had been introduced, Dzik and Wronski observed \cite{Dzik_Wronski_1973} that not only the deductive system $\DS_{l} \bydef \ds{\Ax^i + ((p \to q) \lor (q \to p))}{\MP}$ is structurally complete, but all its extensions are structurally complete too; that is, $\DS_l$ is \textit{hereditarily structurally complete} (and we denote this by $\DS_l \in \HSCpl$). In \cite{Citkin_1978} Citkin had obtained a criterion of hereditary structural completeness for the deductive systems $\ds{\Ax}{\MP}$, where $\Ax^i \subseteq \Ax$. Later, in \cite{Rybakov_Hereditary_1995} Rybakov had proven a similar criterion for normal extensions of modal logic $\KF$. More recently, Olson, Raftery and van Alten (see \cite{Olson_Raftery_Alten_2008} ) had established hereditarily structural completeness for a range of substructural logics. For fuzzy logics and their fragments the hereditary structural completeness was studied by Cintula and Metcalfe \cite{Cintula_Metcalfe_2009}. 
In \cite{Slomczynska_Algebraic_2012} S{\l}omczy{\'n}ska proved that $\{\leftrightarrow,\neg\neg\}$-fragment of intuitionistic propositional calculus is hereditarily structural complete.

Let us note that admissibility of a rule in a deductive system depends exclusively on the logic of the system: a rule $\Gamma/B$ is admissible in a deductive system $\DS$ if and only if for every (uniform) substitution $\sigma$ (of formulas for variables) we have $\sigma(B) \in \LogDS$ as long as $\sigma(A) \in \LogDS$ for every $A \in \Gamma$. Hence, if $\ruleR$ is admissible in $\ds{\Ax}{\Rules}$, it is admissible in $\ds{\Ax}{\Rules + r}$ too, and $\LogL(\Ds) = \LogL(\ds{\Ax}{\Rules + \ruleR})$. Thus, for every deductive system $\DS \bydef \ds{\Ax}{\Rules}$ there is a structurally complete deductive system $\SCDS$ having the same logic as $\DS$ and in which all admissible in $\DS$ rules are derivable: we can take $\SCDS = \ds{\Ax}{\SC{\Rules}}$, where $\SC{\Rules}$ is the set of all rules admissible in $\DS$. Naturally, $\SCDS$ is structurally complete, and we say that $\SCDS$ is a \textit{structural completion}\footnote{In \cite{Rybakov_Book} the term "admissible closure" is used.} of $\DS$. Clearly, a deductive system $\DS$ is structurally complete if and only if $\DS = \SCDS$. Moreover, given a logic $\LogL$, we can consider a deductive system $\SC{\LogL} \bydef \ds{\LogL}{\SC{\Rules}}$, where $\SC{\Rules}$ is the set of all rules admissible in $\LogL$, and we say that $\SC{\LogL}$ is a \textit{structural completion of logic} $\LogL$.

As we know, the structural completion of a logic is, of course, structurally complete. On the other hand, every hereditarily structural complete deductive system $\DS$ is a structural completion of its logic $\LogL(\DS)$. Thus, if we want study the hereditarily structural complete deductive systems, it is natural to ask structural completions of which logics are hereditarily structural complete. In this paper we focus primarily on superintuitionistic logics, that is on logics extending intuitionistic propositional logic $\Int$ . 

First (in Section \ref{sec-general}) we consider some general properties of hereditarily structurally complete deductive systems. Then, in Section \ref{sec-admder}, we recall definitions and properties of admissible and derivable rules. Section \ref{sec-SCpl} is dedicated to studying hereditarily structurally complete deductive systems, and here we prove the main theorem (Theorem \ref{th-HSCplB}) that establishes a link between hereditary structural completeness and inheritance of the bases of admissible rules. And then (in Section \ref{sec-Ext}) we consider some well known superintuitionistic logics from the standpoint of hereditary structural completeness of their structural completions. Some results are summarized in the Table \ref{table1} (where $\LogL \in \HSCpl$ means $\ds{\LogL}{\MP} \in \HSCpl$).

\begin {table}[h] 
\begin{tabular}{lllll}
\hline
$\LogL$&Description&$\LogL \in \HSCpl$ & $\SC{\LogL} \in \HSCpl$ & Definition\\
\hline 
$\Int$ & (intuitionistic logic) & No & Yes &\\
$\mathsf{LC}$ &(G\"{o}del - Dummett logic)& Yes & Yes&\cite[Table 4.1]{Chagrov_Zakh} \\
$\BDn$ & Logic of finite of depth $n$ & No $(n > 1)$ & No $(n > 3)$&\cite[Table 4.1]{Chagrov_Zakh} \\
$\mathsf{G_n}$ & (G\"{o}del logics)& Yes & Yes & $\mathsf{LC} \cap \mathsf{BD}_{n}$\\
$\mathsf{KC}$ &(Yankov logic) & No  & Yes & \cite[Table 4.1]{Chagrov_Zakh}\\
$\mathsf{KP}$ &(Kreisel-Putnam logic) & No & No& \cite[Table 4.1]{Chagrov_Zakh} \\
$\mathsf{ML}$& (Medvedev logic) & No  & No & \cite[Section 2.9]{Chagrov_Zakh}\\
$\mathsf{V}$ &(logic of Visser's rules)& Yes & Yes & \cite{Iemhoff_Intermediate_2005}\\
$\mathsf{RN}$ &(logic of Rieger-Nishimura ladder)& No & Yes & \cite{Bezhanishvili_N_G_de_Jongh_2008}\\
$\mathsf{Sm}$ &(Smetanich logic)& Yes & Yes & \cite[Table 4.1]{Chagrov_Zakh}\\
\hline
\end{tabular}
\caption {Intermediate Logics and Their Structural Completions.} \label{table1}
\end{table}

\section{Deductive Systems and Consequence Relations} \label{sec-general}

\subsection{Deductive Systems}
In this section we recall the basic properties of (propositional) deductive systems and their links to consequence relations. 


\textbf{Language.} A \textit{propositional language} $\Lang$ is understood as a finite set of connectives with specified finite arities. The notion of $\Lang$-formula over a fixed countably infinite set of propositional variables $\Pvar$ is defined in the usual manner. $\For$ denotes the set of all $\Lang$-formulas. A mapping $\sigma: \Pvar \to  \For$ is called a \textit{substitution}. If $\sigma$ is a substitution and $A$ is a formula,  by $\sigma(A)$ we denote a result of simultaneous replacement of each occurring of variable $p$ in $A$ with $\sigma(p)$. If $\Gamma$ is a set of formulas, by $\sigma(\Gamma)$ we denote a set $\set{\sigma(A)}{A \in \Gamma}$.	We say that a formula $B$ is a \textit{substitution instance of a formula} $A$ if there is a substitution $\sigma$ such that $B = \sigma(A)$.

\textbf{Inference Rules. }An ordered couple $\langle \Gamma,A \rangle$, where $\Gamma \subseteq \For$ is a finite (maybe empty) set of formulas and $A \in \For$ is a formula, is called a \textit{(structural inference) rule}. We use a more custom notation: $\Gamma/A$. Formulas $\Gamma$ are called \textit{premises}, while formula $A$ is called a \textit{conclusion}. The set of all rules we denote by $\AllRules$.

If $\ruleR \bydef \Gamma/A$ is a rule and $\sigma$ is a substitution, then $\sigma(\Gamma)/\sigma(A)$ is said to be a $\sigma$-\textit{substitution instance of rule} $\ruleR$ (denoted by $\sigma(\ruleR)$) and we omit reference to a particular substitution when no confusion arises. 

\textbf{Deductive System.} An ordered couple $\Ds$, where $\Ax \subseteq \For$ and $\Rules \subseteq \AllRules$ are non-empty sets respectively of formulas and of rules, is called a \textit{deductive system}. Let $\DS = \Ds$ be a deductive system. Then $\Ax$ is a set of \textit{axioms} of $\DS$ and $\Rules$ is a set of \textit{rules of }$\DS$. And, given a finite set of formulas $\Gamma$ and a formula $A$, a sequence of formulas $B_1,\dots,B_n$ is called an $\DS$\textit{-inference} of $A$ from $\Gamma$ if the following holds
\begin{itemize}
\item[(i1)] $B_n = A$;
\item[(i2)] for all $i=1,\dots,n$ either $B_i \in \Gamma$, or $B_i$ is a substitution instance of an axiom from $\Ax$, or there is a substitution instance $\Delta/B_i$ of a rule from $\Rules$ such that $\Delta \subseteq \{B_1,\dots,B_{i-1}\}$. 
\end{itemize} 
If there is an $\DS$-inference of $A$ from $\Gamma$, we say that $A$ is an $\DS$\textit{-consequence} of $\Gamma$ and we denote this by $\Gamma \vdash_\DS A$. If $\Gamma = \varnothing$, we write $\vdash_\DS A$ and say that $A$ is $\DS$-\textit{derivable} or that $A$ is an $\DS$-\textit{theorem}. The set of all $\DS$-theorems is called a \textit{logic of deductive system} $\DS$, and we denote this set by $\LogL(\DS)$.

\begin{example} As examples, we will often use the following two deductive systems: $\IPC \bydef \ds{\Ax^i}{\MP}$ and $\CPC \bydef \ds{\Ax^c}{\MP}$, where $\Ax^i$ and $\Ax^c$ are sets of axiom schemes of intuitionistic and classical propositional calculi from \cite[Section 1]{{Mendelson_Introduction_2015}}, and $\MP$ is Modus Ponens.
\end{example}

\subsection{Consequence Relations}


A (finitary structural) \textit{consequence relation} is a relation $\vdash$ between finite sets of formulas and formulas satisfying the following conditions: for any finite sets of formulas $\Gamma,\Gamma'$ and each formula $A \in \For$
\begin{itemize}
\item[(a)] $A \vdash A$
\item[(b)] if $\Gamma \vdash A$ then $\Gamma \cup \Gamma' \vdash A$
\item[(c)] if $\Gamma \vdash B$ for every $B \in \Gamma'$ and $\Gamma' \vdash A$, then $\Gamma 
\vdash A$
\item[(d)] if $\Gamma \vdash A$ the $\sigma(\Gamma) \vdash \sigma(A)$ for every substitution $\sigma$.
\end{itemize}

Let us note that relation $\vdash_\DS$ for a given deductive system $\DS$ satisfies the above definition and, hence,  $\vdash_\DS$ is a \textit{consequence relation defined by }$\DS$. On the other hand, given a consequence relation $\vdash$, one can take a deductive system $\DS \bydef \ds{\set{A \in \For}{\vdash A}}{\set{\Gamma/B \in \AllRules}{\Gamma \vdash B}}$ and verify that $\vdash_\DS \ = \ \vdash$, that is, every consequence relation can be defined by a deductive system. 

Given two consequence relations $\vdash$ and $\vdash'$ we say that $\vdash'$ \textit{extends} $\vdash$, or that $\vdash'$ is an \textit{extension} of $\vdash$, in written $\vdash \ \leq \ \vdash'$, if $\vdash \ \subseteq \ \vdash'$. We say that $\vdash'$ \textit{strongly extends} $\vdash$, or that $\vdash'$ is a \textit{proper extension} of $\vdash$, in written $\vdash \ < \ \vdash'$, if $\vdash \ \subsetneq \ \vdash'$. It is clear that the set of all extensions of a given consequence relation is closed under arbitrary meets, and, therefore, it forms a complete lattice. An extension $\vdash'$ of a consequence relation $\vdash$ is said to be \textit{axiomatic} if there is a set of formulas $\Gamma$ such that $\vdash'$ is the smallest extension of $\vdash$ having every formula from $\Gamma$ as a theorem. In terms of deductive systems, $\vdash'$ is an axiomatic extension of $\vdash_\DS$, where $\DS \bydef \Ds$, if $\vdash'$ can be defined by a deductive system $\ds{\Ax + \Gamma}{\Rules}$ for some $\Gamma \subseteq \For$ (comp. \cite{Olson_Raftery_Alten_2008}).    

\begin{example} Then $\vdash_{\CPC}$ is an axiomatic extension of $\vdash_{\IPC}$, because $\vdash_{CPC}$ can be defined by $\ds{\Ax^i + (\neg\neg p \to p)}{\MP}$.
\end{example}

Let $\DS_1$ and $\DS_2$ be deductive systems. We say that $\DS_2$ is logical extension of $\DS_1$ (in symbols $\DS_1 \LogExt \DS_2$) if $\LogL(\DS_1) \subseteq \LogL(\DS_2)$; and we say that $\DS_2$ is a \textit{deductive extension} of $\DS_1$ (and we denote this by $\DS_1 \DedExt \DS_2$) if $\vdash_{\DS_1} \ \subseteq \ \vdash_{\DS_2}$. We also say that $\DS_1$ and $\DS_2$ are \textit{logically equal} (in symbols $\DS_1 \LogEqv \DS_2$) if $\LogL(\DS_1) = \LogL(\DS_2)$, and $\DS_1$ and $\DS_2$ are \textit{deductively equal} and we write $\DS_1 \DedEqv \DS_2$, if $\vdash_{\DS_1}\  = \ \vdash_{\DS_2}$.

\begin{example} For instance, $\IPC \DedExt \CPC$. Moreover, $\CPC$ is an axiomatic extension of $\IPC$: one can take $\Gamma = \{ (\neg\neg p \to p)\}$.
\end{example}

Let us note the following, rather simple property that we will need in the sequel.

\begin{prop}\label{pr-addtheorems} Let $\DS \bydef \Ds$ be a deductive system. Then
\begin{equation}
\Ds \DedEqv \ds{\LogL(\DS)}{\Rules}.  \label{eq-addtheorems}
\end{equation}
\end{prop}
This proposition simply means that extending the set of axioms by formulas derived in $\DS$ does not change the consequence relation defined by $\DS$.
\begin{proof} The proof immediately follows from the definition of inference. Indeed, let $A_1,\dots,A_k,\dots,A_n$ be an inference of $A_n$ from a set of formulas $\Gamma$ and $A_k \in \LogL(\DS)$. Then there is an inference $B_1,\dots, B_m, A_k$ of $A_k$ from the empty set of formulas. Immediately from the definition of inference we can see that $A_1,\dots,B_1,\dots, B_m, A_k,\dots,A_n$ is an inference of $A_n$ from $\Gamma$. Thus, any inference in $\ds{\LogL(\DS)}{\Rules}$ can be converted into an inference in $\DS$.
\end{proof}

\section{Admissible and Derivable Rules} \label{sec-admder}

The goal of this section is to recall the notions of admissibility and derivability of rules in deductive systems. 

\subsection{Admissible Rules} We start by recalling the notion of a rule admissible in a given deductive system.

\begin{definition} Let $\DS \bydef \Ds$ be a deductive system. A rule $\Rule$ is called \textit{admissible in} $\DS$, or in the corresponding consequence relation $\vdash_\DS$, if logic $\LogL(\DS)$ is closed under $\ruleR$, that is, $\LogL(\Ds) = \LogL(\ds{\Ax}{\Rules + \ruleR})$ (or $\Ds \LogEqv \ds{\Ax}{\Rules + \ruleR}$).
\end{definition}

If $\DS = \Ds$ is a deductive system, by $\SC{\Rules}(\DS)$ we denote the set of all rules admissible in $\DS$. Since adding an admissible rule to a deductive system does not change the logic of this system and all rules from $\Rules$ are trivially admissible in $\DS$ (that is, $\Rules \subseteq \SC{\Rules}(\DS)$), we have
\begin{equation} 
\DS \LogEqv \ds{\Ax}{\SC{\Rules}}. \label{eq-admnotext}
\end{equation}

\begin{example} Rule $\ruleR \bydef \neg (p \to (q \lor r))/((\neg p \to q) \lor (\neg p \to r))$ - a Harrop rule - is admissible in $\IPC$ because, as it follows from \cite{Harrop_Concerning_1960}, $\LogL(\ds{\Ax^i}{\MP}) = \LogL(\ds{\Ax^i}{\MP + \ruleR})$.
\end{example}

The following Proposition gives a well known alternative intrinsic characterization of admissibility.

\begin{prop} A rule $\ruleR \bydef \Gamma/A$ is admissible in a deductive system $\DS$ if and only if for every $\sigma$-substitution instance of $\ruleR$
\begin{equation}
\sigma(\Gamma) \subseteq \LogL(\DS) \text{ yields } \sigma(A) \in \LogL. \tag{ADM} \label{ADM} 
\end{equation} 
\end{prop} 
\begin{proof} Let $\DS \bydef \ds{\Ax}{\Rules}$, $\DS' \bydef \ds{\Ax}{\Rules + \ruleR}$. Then, obviously, $\LogL(\DS) \subseteq \LogL(\DS')$ and we need to show that \eqref{ADM} is equivalent to $\LogL(\DS') \subseteq \LogL(\DS)$.
Assume that \eqref{ADM} holds. By simple induction on length of inference one can demonstrate that any $\DS'$-inference from $\varnothing$ is, by the same token, an $\DS$-inference from $\varnothing$. Thus, $\LogL(\DS') \subseteq \LogL(\DS)$. 

Now, assume that \eqref{ADM} does not hold and suppose $\Gamma = \{A_1,\dots,A_n\}$. Then for some substitution $\sigma$, $\sigma(\Gamma) \in \LogL(\DS)$ and $\sigma(A) \notin \LogL(\DS)$. Due to $\sigma(A_i) \in \LogL(\DS)$ for every $i=1,\dots,n$, by the definition of $\LogL(\DS)$, for each formula $\sigma(A_i)$ there is an $\DS$-inference $I_i$ of $A_i$ from $\varnothing$. 
Clearly, every $I_i, i=1,\dots,n$ is at the same time an $\DS'$-inference of $\sigma(A)$ from $\varnothing$. Hence, $I_1,\dots,I_b,\sigma(A)$ is an $\DS'$-inference of $\sigma(A)$ from $\varnothing$, which means that $\sigma(A) \in \LogL(\DS')$. Thus, $\sigma(a) \in \LogL(\DS')$ and $\sigma(A) \notin \LogL(\DS)$, i.e. $\LogL(\DS) \subset \LogL(\DS')$.

The case $\Gamma = \varnothing$ is trivial.
\end{proof}

Thus, admissibility of rules depends only on logic, that is, the following holds.

\begin{prop} \label{pr-admlogequal} Let $\DS_0$ and $\DS_1$ be logically equal deductive system. Then a rule $\ruleR$ is admissible in $\DS_0$ if and only if $\ruleR$ is admissible in $\DS_1$, that is,
\begin{equation}
\SC{\Rules}(\DS_0) = \SC{\Rules}(\DS_1), \label{eqpr-admlogequal}
\end{equation}
where $\SC{\Rules}(\DS_i)$ denotes the set of all rules admissible in $\DS_i, i=0,1$.
\end{prop}

Since deductive equality yields logical equality, due to the above Proposition, we can speak about admissibility of a rule for a consequence relation, because, if a rule $\ruleR$ is admissible in $\DS$, then $\ruleR$ is admissible in every deductively equal to $\DS$ system, that is, in every deductive system defining $\vdash_\DS$. 

\begin{prop} \label{pr-DSadmissible} Let $\vdash$ be a consequence relation and $\ruleR$ be a rule. Then the following is equivalent
\begin{itemize}
\item[(a)] $\ruleR$ is admissible for $\vdash$ ;
\item[(b)] $\ruleR$ is admissible in some deductive system defining $\vdash$;
\item[(c)] $\ruleR$ is admissible in every deductive system defining $\vdash$.
\end{itemize}
\end{prop} 

If $\DS \bydef \Ds$ is a deductive system, by $\SCR$ we denote the set of all rules admissible in $\DS$, and by $\admDS$ we denote the consequence relation defined by deductive system $\ds{\Ax}{\SCR}$. Let us observe that $\admDS$ is the greatest consequence relation having $\LogL(\DS)$ as its set of theorems.

\subsection{Derivable Rules} Let $\DS \bydef \Ds$ be a deductive system. 

\begin{definition} If $\ruleR \bydef \Gamma/A$ is a rule and $\Gamma \vdash_\DS A$, we say that $\ruleR$ is \textit{derivable} in $\DS$, or that $\ruleR$ is $\DS$\textit{-derivable}. $\Rules(\DS)$ denotes a set of all $\DS$-derivable rules.
\end{definition}

\begin{example} Every admissible in $\CPC$ rule is derivable in $\CPC$, while Harrop rule is admissible and not derivable in $\IPC$. 
\end{example}

It is clear that every derivable in $\DS$ rule is admissible in $\DS$, that is, $\Rules(\DS) \subseteq \SCR$, but not necessarily vice versa, as we see from the above example.

\begin{prop}\label{pr-addrules} Let $\DS \bydef \Ds$ be a deductive system. Then
\begin{equation}
\Ds \DedEqv \ds{\Ax}{\Rules(\DS)}.  \label{eq-addrules}
\end{equation}
\end{prop}
This proposition simply means that extending the set of rules by the derivable in $\DS$ rules does not change the consequence relation defined by $\DS$.
\begin{proof} The proof is similar to the proof of Proposition \ref{pr-addtheorems}.
\end{proof}

The definition of derivability can be rephrased in terms of consequence relations: given a consequence relation $\vdash$, a rule $\Gamma/A$ is $\vdash$-derivable if $\Gamma \vdash A$. Let us note that the following holds.

\begin{prop} \label{pr-DSderivable} Let $\vdash$ be a consequence relation and $\ruleR$ be a rule. Then the following is equivalent
\begin{itemize}
\item[(a)] $\ruleR$ is $\vdash$-derivable;
\item[(b)] $\ruleR$ is derivable in some deductive system defining $\vdash$;
\item[(c)] $\ruleR$ is derivable in every deductive system defining $\vdash$.
\end{itemize}
\end{prop} 
The proof is easy and it is left for the reader.

Thus, if $\ruleR$ is a rule and $\DS \bydef \Ds$ is a deductive system, $\ruleR$ is admissible in $\DS$ if and only if $\ds{\Ax}{\Rules + \ruleR} \LogEqv \Ds$, and $\ruleR$ is derivable in $\DS$ if and only if $\ds{\Ax}{\Rules + \ruleR} \DedEqv \Ds$.

\subsection{Base of Admissible Rules}

One of the common ways of defining the set of all rules admissible in a given deductive system $\DS \bydef \Ds$ is to present a base, that is, a set of admissible rules from which every admissible in $\DS$ rule can be derived. 

\begin{definition} Suppose $\DS \bydef \Ds$ is a deductive system. A set of rules $\Rules'$ is a \textit{relative to} $\DS$ \textit{base of admissible rules} if $\ds{\Ax}{\Rules + \Rules'} \DedEqv \ds{\Ax}{\SCR}$. And $\Rules'$ is a \textit{base of admissible in }$\DS$ \textit{rules} if $\ds{\LogL(\DS)}{\Rules'} \DedEqv \ds{\LogL(\DS)}{\SCR}$. 
\end{definition}

The following simple proposition shows the relations between relative bases and bases.

\begin{prop}\label{pr-relbasetobase}
Let $\DS \bydef \Ds$ be a deductive system  and $\Rules'$ be a relative to $\DS$ base of admissible rules. Then $\Rules + \Rules'$ forms a base of admissible in $\DS$ rules. 
\end{prop}
\begin{proof} Indeed, by \eqref{eq-addtheorems}, $\ds{\Ax}{\Rules + \Rules'} \DedEqv \ds{\LogL(\DS)}{\Rules + \Rules'}$. On the other hand, by assumption and by \eqref{eq-addtheorems}, keeping in mind that deductive equality yields logical equality, we have $\ds{\Ax}{\Rules + \Rules'} \DedEqv \ds{\Ax}{\SCR} \DedEqv \ds{\LogL(\DS)}{\SCR}$.
\end{proof}

\begin{example} As it had been observed in \cite{Iemhoff_Admissible_2001}, the rules $V_n, n=1,2,\dots$ form relative to $\IPC$ base for admissible (in $\IPC$) rules:
\begin{equation}
V_n \bydef r \lor \bigwedge_{i=1}^n(p_i \to q_i) \to (p_{n +1} \lor p_{n+2})/ r \lor \bigvee_{j=1}^{n+2}(\bigwedge_{i=1}^n(p_i \to q_i) \to p_j).  
\end{equation}
Rules $V_n$ are knows as the Visser's rules. The set of the Visser's rules together with Modus Ponens forms also a base of admissible in $\IPC$ rules.
\end{example}

Naturally, any two logically (and, therefore, any two deductively) equal systems share the base of admissible rules. In other words, the base of admissible rules depends only on logic and does not depend on a particular deductive system defining this logic.

\begin{prop}\label{pr-logadmbase} Let $\DS_0, \DS_1$ be logically equal deductive systems. Then a set of rules $\Rules$ is a base of admissible in $\DS_0$ rules if and only if $\Rules$ is a base of admissible in $\DS_1$ rules
\end{prop}
The proof is trivial: by assumption, $\DS_0 \LogEqv \DS_1$ , that is, $\LogL(\DS_0) = \LogL(\DS_1)$.

\section{Structural Completeness} \label{sec-SCpl}

In this section we recall the notions of structural and hereditary structural completeness, and we prove the main theorem (Theorem \ref{th-HSCplB}) that establishes a link between hereditary structural completeness and inheritance of a base of admissible rules.

\subsection{Structural Completeness: Definition}\label{SCpl} The notion of structural completeness was introduced in \cite{Pogorzelski_Structural_1971} and is central for our research. 

\begin{definition} \cite{Pogorzelski_Structural_1971}\label{def-SCpl} A deductive system $\DS$ is said to be \textit{structurally complete} if $\Rules(\DS) = \SC{\Rules}(\DS)$, i.e. if every admissible in $\DS$ rule is derivable in $\DS$. 
\end{definition}

\begin{example} $\CPC$ is structurally complete, while $\IPC$ is not.
\end{example}

Immediately from Propositions \ref{pr-DSadmissible} and \ref{pr-DSderivable} we obtain the following.

\begin{prop} \label{pr-DSSCpl} Let $\vdash$ be a consequence relation. Then the following is equivalent

\begin{tabular}{rl}
(a) &$\vdash$ is structurally complete;\\
(b) &Some deductive system defining $\vdash$ is structurally complete; \\
(c) &Every deductive system defining $\vdash$ is structurally complete.
\end{tabular}
\end{prop} 

The next Proposition gives an alternative intrinsic definition of structural completeness, and often it is used as a definition (see e.g. \cite{Bergman_Structural_1991,Olson_Raftery_Alten_2008,Cintula_Metcalfe_2009}). 

\begin{prop} \label{pr-SCplCR} A consequence relation $\vdash$ is \textit{structurally complete} if and only if every its proper extension $\vdash'$ contains the new theorems, i.e.
\begin{equation}
\vdash \ < \ \vdash' \text{ entails } Th(\vdash) \subset Th(\vdash').  \label{eqpr-SCplCR} \tag{SC}
\end{equation}
\end{prop} 
\begin{proof}
Suppose that $\vdash$ is a structurally complete consequence relation and let $\vdash'$ be a consequence relation and $\vdash \ < \ \vdash'$. Due to $\vdash$ is structurally complete, every admissible, that is every preserving $Th(\vdash)$, rule is $\vdash$-derivable. Hence, $\vdash$ cannot have the same set of theorems as $\vdash$. Thus, \eqref{SCpl} holds.

Conversely, assume that \ref{SCpl} holds and $\Rules$ is a set of all derivable in $\vdash$ rules. Then $\vdash \ = \ \vdash_\DS$, where $\DS \bydef \ds{Th(\vdash)}{\Rules}$. If $\DS$ is not structurally complete, there would be an admissible in $\DS$ but not $\DS$-derivable rule $\ruleR$. Consider deductive system $\DS^\ruleR \bydef \ds{Th(\vdash)}{\Rules + \ruleR}$. Due to $\ruleR$ is admissible in $\DS$, we have $Th(\DS) = Th(\DS^\ruleR)$. Due to $\ruleR$ is not $\DS$-derivable, we have $\vdash_\DS \ < \ \vdash_{\DS^\ruleR}$.
\end{proof}

\subsection{Structural Completions}

It is worth noting that every deductive system $\DS \bydef \Ds$ can be extended to a logically equal structurally complete deductive system. Indeed, denote by $\SC{\Rules}(\DS)$ the set of all rules admissible in $\DS$, and take deductive system $\SCDS \bydef \ds{\Ax}{\SCR}$. It is clear that $\DS \LogEqv \SCDS$. 

\begin{definition} 
Deductive system $\SCDS \bydef \ds{\Ax}{\SCR}$ is called a \textit{structural completion} of $\DS$.
\end{definition}

Let us observe that  $\SCDS$ is the greatest relative to $\DedExt$ system among deductive systems logically equal to $\DS$. More precisely, the following holds.
 
\begin{prop} \label{pr_CompGr}(comp. \cite[Proposition 1.2]{Bergman_Structural_1991}) Let $\DS$ and $\DS'$ be deductive systems. Then
\begin{equation}
\DS \LogEqv \DS' \text{ entails }\DS' \DedExt \SCDS. \label{eqpr_CompGr}
\end{equation}
\end{prop}
\begin{proof} Let $\DS \bydef \Ds$ and $\DS' \bydef \ds{\Ax'}{\Rules'}$. Then, due to every derivable rule is admissible and by \eqref{eq-addtheorems}, we have
\[
\DS' = \ds{\Ax'}{\Rules'} \DedExt \ds{\Ax'}{\SC{\Rules}(\DS')}  \DedEqv \ds{\LogL(\DS')}{\SC{\Rules}(\DS')} \DedEqv \SC{\DS'}. 
\]
Let us observe, that $\DS \LogEqv \DS'$ yields $\LogL(\DS) = \LogL(\DS')$, and that, by \eqref{eqpr-admlogequal}, $\SC{\Rules}(\DS) = \SC{\Rules}(\DS')$. Thus,
\[
\SC{\DS'} \DedEqv \ds{\LogL(\DS')}{\SC{\Rules}(\DS')} \DedEqv \ds{\LogL(\DS)}{\SC{\Rules}(\DS)} \DedEqv \SC{\DS}.
\]
\end{proof}

\subsection{Hereditary Structural Completeness} 

\begin{definition} A structurally complete deductive system $\DS$ is \textit{hereditarily structurally} complete if every its deductive extension is structurally complete (comp. \cite[Section 5.4]{Rybakov_Book}).
\end{definition}

Immediately from the definition it follows that any deductive extension of a hereditarily structurally complete deductive system is hereditarily structurally complete. At the same time, there are structurally complete deductive systems which are not hereditarily structurally complete (see e.g. \cite{Citkin_1978}).

The following theorem gives some alternative views at hereditary structural completeness.

\begin{theorem}\label{th-HScpl}\cite[Theorem 2.6]{Olson_Raftery_Alten_2008} Let $\DS \bydef \Ds$ be a structural complete deductive system Then the following is equivalent:

\begin{tabular}{rl}
(a) & $\DS$ is hereditarily structural complete;\\
(b) &  Every axiomatic extension of $\DS$ is structural complete;\\
(c) & Every deductive extension of $\DS$ is axiomatic.
\end{tabular}
\end{theorem}

Let us note that if a deductive system $\DS$ is hereditarily structural complete, then $\DS$ is structural completion of a deductive system. Thus, it is natural to ask for a given deductive system $\DS$ whether its structural completion $\SCDS$ is hereditarily complete. The next theorem gives some necessary and sufficient conditions of hereditarily structural completeness of the structural completion of a deductive system. 

\begin{theorem} \label{th-HSCplB}Let $\DS \bydef \Ds$ be a deductive system and $\Rules_b$ be a relative base of admissible in $\DS$ rules. Then $\SCDS$ is hereditarily structurally complete if and only if $\Rules_b$ forms a relative basis of admissible rules in every deductive extension of $\DS$ in which rules $\Rules_b$ are admissible.  
\end{theorem}
\begin{proof}
Suppose $\SCDS$ is a hereditarily structurally complete deductive system and $\DS' \bydef \ds{\Ax'}{\Rules'}$ is a deductive extension of $\DS$ admitting all rules from $\Rules_b$. Assume for contradiction that $\Rules_b$ is not a relative base of admissible rules of 
$\DS$. We will demonstrate that in this case, if $\DS_b \bydef \ds{\Ax'}{\Rules' + \Rules_b}$, then

\begin{tabular}{rl}
(a)& $\DS_b$ is a deductive extension of $\SCDS$, i.e. $\SCDS \DedExt \DS_b $;\\
(b)&$\DS_b$ is not structurally complete. 
\end{tabular}

\noindent Thus (a) and (b)  entail that $\SCDS$ has a non-structurally complete deductive extension and, therefore, $\SCDS$ is not $\HSCpl$ contrary to the assumption.

\textbf{Proof of (a).} By assumption, $\Rules_b$ is a relative basis for $\DS$, that is,
\begin{equation}
\ds{\Ax}{\Rules + \Rules_b} \DedEqv \SC{\DS}. \label{eqth-HSCplB1}
\end{equation}
By \eqref{eq-addtheorems} and \eqref{eq-addrules} we also have
\begin{equation}
\DS' = \ds{\Ax'}{\Rules'} \DedEqv \ds{\LogL(\DS')}{\Rules'} \DedEqv \ds{\LogL(\DS')}{\Rules(\DS')}. \label{eqth-HSCplB2}
\end{equation}
Recall that $\DS'$ is a deductive extension of $\DS$, therefore,
\begin{equation}
\Ax \subseteq \LogL(\DS') \text{ and } \Rules \subseteq \Rules(\DS'). \label{eqth-HSCplB2}
\end{equation}
At the same time, $\DS_b$ was obtained from $\DS'$ by adding to $\DS'$ new rules, hence, $\DS_b$ is a deductive extension of $\DS'$, and, therefore, we can extend \eqref{eqth-HSCplB2}:
\begin{equation}
\Ax \subseteq \LogL(\DS') \subseteq \LogL(\DS_b) \text{ and } \Rules \subseteq \Rules(\DS') \subseteq \Rules(\DS_b). \label{eqth-HSCplB3}
\end{equation}
Moreover, due to $\Rules_b \subseteq \Rules(\DS_b)$, from \eqref{eqth-HSCplB3} we have
\begin{equation}
\Ax \subseteq \LogL(\DS_b) \text{ and } \Rules + \Rules_b \subseteq \Rules(\DS_b). \label{eqth-HSCplB4}
\end{equation}
And from \eqref{eqth-HSCplB1}, \eqref{eqth-HSCplB4} and  by \eqref{eq-addtheorems}, \eqref{eq-addrules}
\begin{equation}
\SCDS \DedEqv \ds{\Ax}{\Rules + \Rules_b} \DedExt \ds{\LogL(\DS_b)}{\Rules(\DS_b)} \DedEqv \ds{\Ax'}{\Rules' + \Rules_b} = \DS_b, \label{eqth-HSCplB5}
\end{equation}
that is, $\DS_b$ is a deductive extension of $\SCDS$.

\textbf{Proof of (b).} By assumption, rules $\Rules_b$ do not form a relative base for $\DS'$,that is, 
\begin{equation}
\DS_b = \ds{\Ax'}{\Rules' + \Rules_b} \not\DedEqv \ds{\Ax'}{\SC{\Rules}(\DS')}. \label{eqth-HSCplB6}
\end{equation}
Now, we only need to establish that $\DS_b$ and $\ds{\Ax'}{\SC{\Rules}(\DS')}$ have the same logic.
Indeed, by the assumption of the theorem, rules $\Rules_b$ are admissible in $\DS'$, so
\begin{equation}
\ds{\Ax'}{\Rules'} \LogEqv \ds{\Ax'}{\Rules' + \Rules_b} \label{eqth-HSCplB7}
\end{equation}
and, therefore, by \eqref{eqth-HSCplB7} and \eqref{eq-admnotext} 
\begin{equation}
\ds{\Ax'}{\SC{\Rules}(\DS')} \LogEqv \ds{\Ax'}{\Rules'} \LogEqv \ds{\Ax'}{\Rules' + \Rules_b}. \label{eqth-HSCplB8}
\end{equation}

Conversely, suppose that rules $\Rules_b$ form a relative base of admissible rules in every deductive extension of $\DS$ admitting rules $\Rules_b$. We need to show that $\SCDS$ is $\HSCpl$, that is, that its every deductive extension is structurally complete. 

Assume that $\DS' \bydef \ds{\Ax'}{\Rules'}$ is a deductive extension of $\SCDS$. Then by \eqref{eq-addtheorems} and  by \eqref{eq-addrules},
\begin{equation}
 \DS' = \ds{\Ax'}{\Rules'} \DedEqv \ds{\LogL(\DS')}{\Rules(\DS')}. \label{eqth-HSCplB9}
\end{equation}
Let us recall that rules $\Rules_b$ are admissible in $\DS$ and, therefore, $\Rules_b \subseteq \Rules(\SCDS)$ and, by the assumption, $\DS'$ is a deductive extension of $\SCDS$, so, 
\begin{equation}
\Rules_b \subseteq \Rules(\SCDS) \subseteq \Rules(\DS'). \label{eqth-HSCplB10}
\end{equation}
\eqref{eqth-HSCplB10} means, that rules $\Rules_b$ are derivable in $\DS'$, and $\DS'$ is structurally complete due to all rules from a relative base (of admissible in $\DS'$ rules) are derivable in $\DS'$.
\end{proof}

\begin{example} In \cite{Rybakov_Intermediate_1993} Rybakov described the class of all axiomatic extensions of $\IPC$ admitting all rules from $\SC{\IPC}$. In the Section \ref{sec-Int} we discuss this class in more details.  
\end{example}

\section{Hereditarily Structural Complete Extensions of superintuitionistic logics} \label{sec-Ext}

In this section we study hereditary structural completeness of deductive extensions of $\IPC$. All hereditarily structural complete axiomatic extensions of $\IPC$ (that is, the deductive systems of type $\ds{\Ax^i + \Ax}{\MP}$). The set $\mathcal{HSC}$ of hereditarily structural complete axiomatic extensions of $\IPC$ (and of $\KF$ for this matter) has rather nice properties:

\begin{tabular}{rl}
(a)& $\mathcal{HSC}$ is countably infinite;\\
(b)& every deductive system from $\mathcal{HSC}$ is finitely axiomatizable;\\
(c)& $\mathcal{HSC}$ contains the least (relative to $\DedExt$) deductive system.
\end{tabular}

\noindent As we will see, in a general case the situation is more complex, namely, none of the above properties holds (see Corollary \ref{cor-cont} below). Failure of (c) also entails that the criteria similar to ones established in \cite{Citkin_1978} and \cite{Rybakov_Hereditary_1995}, are impossible. In this section, we focus on hereditary structural completions of the standard superintuitionistic logics listed in the Table  \ref{table1}.

\subsection{Hereditary Structural Completeness of Int + Visser Rules} \label{sec-Int}

First, we establish that $\SC{\IPC}$ is hereditarily structurally complete, and then we will consider some deductive extensions of $\SC{\IPC}$.

\begin{theorem} \label{th-SCIPC} Structural completion of $\IPC$ is $\HSCpl$.
\end{theorem}
\begin{proof} It was established in \cite[Theorem 3.20]{Iemhoff_Admissible_2001} that Visser's rules $V_n, n>0$ are admissible in $\IPC$, and it was observed in \cite[Theorem 3.9]{Iemhoff_Intermediate_2005} that Visser's rules form a base of admissible rules in every deductive extension of $\IPC$ which admits them. Hence, we can apply Theorem \ref{th-HSCplB} and complete the proof.  
\end{proof}

Recall that Visser's rules are admissible in $\mathsf{KC}$ and $\mathsf{M}_n$ (see \cite[Theorem 5.1]{Iemhoff_Intermediate_2005}) and  Visser's rules are derivable in $\mathsf{Bd}_1,\mathsf{G}_n,\mathsf{LC}$, and $\mathsf{Sm}$ (see \cite[Theorem 5.3]{Iemhoff_Intermediate_2005}). And we know that any deductive extension of $\HSCpl$ system is hereditarily structurally complete. Hence, we have the following.

\begin{cor} \label{cor-HScplFirst} Structural completions of the following deductive system are $\HSCpl$: $\mathsf{KC}, \mathsf{M}_n,$ $\mathsf{KC}, \mathsf{Bd}_1,\mathsf{G}_n,\mathsf{LC},\mathsf{Sm}$.  
\end{cor}

Rybakov observed \cite[Theorem 7]{Rybakov_Intermediate_1993} that there is continuum many intermediate logics admitting all rules admissible in $\IPC$. Hence, there is continuum many not logically equivalent deductive systems in which all Visser's rules are admissible. Thus, the following holds.

\begin{cor} \label{cor-cont} There is continuum many not logically equivalent $\HSCpl$ deductive systems (extending $\SC{INT}$). Therefore, there are not finitely axiomatizable $\HSCpl$ deductive systems (extending $\SC{INT}$). 
\end{cor} 

In the following section we will prove that there is continuum many structural completions of of superintuitionistic logics that are not hereditarily structurally complete.

\subsection{Hereditary Structural Completeness: The Algebraic View} \label{sec-alg}

Generally speaking, there are two ways to prove that a deductive system $\DS$ is not structurally complete: to present an admissible in $\DS$ and not $\DS$-derivable rule, or to use semantic means. It is known (see e.g. \cite{Olson_Raftery_Alten_2008}) that each (finitely algebraizable in sense of Blok and Pigozzi \cite{Blok_Pigozzi_Algebraizable_1989}) deductive system corresponds to a quasivariety of algebras which are models for this system. In this Section we use the second approach to show that there is continuum many superintuitionistic logics, whose structural completion is not $\HSCpl$. But first, we need to recall some notions and facts from the theory of quasivarieties.

\textbf{Basic facts from theory of quasivarieties.} Let us recall some basic notions about models of superintuitionistic logics. As usual, we use Heyting algebras\footnote{Or frames representing them - see e.g. \cite{Rybakov_Book}, where Heyting algebras are called "pseudo-Boolean algebras".} as models for superintuitionistic logics. A bounded distributive lattice $\langle \alg{A}; \land,\lor,\zero,\one \rangle$ with  relative pseudocomplementation $\to$ is called a \textit{Heyting algebra} (see e.g. \cite[Section II]{Burris_Sanka}), and we abbreviate $\alg{a} \to \zero$ as $\neg \alg{a}$. A formula $A$ is \textit{refuted in} a given (Heyting) algebra $\Alg{A}$, if there is a valuation $\nu$ in $\Alg{A}$ such that $\nu(A) \neq \one$. Otherwise $A$ is said to be \textit{valid in} $\Alg{A}$. A rule $\ruleR \bydef A_1,\dots,A_n/B$ is \textit{refuted in} a given algebra $\Alg{A}$, if there is a valuation $\nu$ in $\Alg{A}$ such that $\nu(A_i) = \one$ for all $i=1,\dots,n$, but $\nu(B) \neq \one$. Otherwise $\ruleR$ is said to be \textit{valid in} $ \Alg{A}$. Given a set of formulas $\Gamma$ and a formula $A$ (or set of rules $\Rules$ and a rule $\ruleR$) we say that an \textit{algebra} $\Alg{A}$ \textit{separates} $A$ \textit{from} $\Gamma$ (or that $\Alg{A}$ separates $\ruleR$ from $\rules$), if $\Alg{A}$ refutes $A$ while all formulas from $\Gamma$ are valid in $\Alg{A}$ (if $\Alg{A}$ refutes $\ruleR$, while all rules from $\Rules$ are valid in $\Alg{A}$). 

With each superintuitionistic deductive system $\DS \bydef \Ds$ one can associate a quasivariety $\var{Q}(\DS)$ of all algebras in which all axioms and all rules of $\DS$ are valid\footnote{All necessary information about quasivarieties the reader can find in \cite{GorbunovBookE}.} . Moreover, given two deductive systems $\DS_1,\DS_2$,
\[
\DS_1 \DedExt \DS_2 \text{ if and only if } \var{Q}(\DS_1) \supseteq \var{Q}(\DS_2).
\]
and, hence,
\[
\DS_1 \DedEqv \DS_2 \text{ if and only if } \var{Q}(\DS_1) = \var{Q}(\DS_2).
\]
And for every quasivariety $\qvar$ there is a deductive system $\DS$ such that $\qvar = \var{Q}(\DS)$.

Let $\DS$ be a deductive system, $\qvar(\DS)$ be a corresponding quasivariety, and $\FAlg{\qvar(\DS)}{\omega}$ be a free algebra of quasivariety $\Qvar{\DS}$. Then (\cite{Bergman_Structural_1991}[Proposition 2.3]),
\begin{equation}
 \Qvar{\SCDS} = \Qvar{\FAlg{\qvar(\DS)}{\omega}}.
\end{equation}
If $\qvar$ is a quasivariety, we say that $\SC{\qvar} \bydef \Qvar{\FAlg{\qvar}{\omega}}$ is a \textit{structural completion of quasivariety} $\qvar$.

A quasivariety $\qvar$ is \textit{primitive} \cite{GorbunovBookE} if each its subquasivariety is a relative variety, that is, for each subquasivariety $\qvar' \subseteq \qvar$ there is a variety $\vvar$ of such that $\qvar' = \qvar \cap \vvar$.

There is correspondence between hereditarily structurally complete deductive systems and primitive quasivarieties.

\begin{prop} (Comp. e.g. \cite[Corollary 7.15]{Olson_Raftery_Alten_2008}) \label{pr_HScplPrim} A deductive system $\DS$ is hereditarily structurally complete if and only if $\var{Q}(\DS)$ is primitive. 
\end{prop} 
Thus, if for a deductive system $\DS$ its structural completion is $\HSCpl$ if and only if $\SC{\qvar}(\DS)$ is primitive.

Let $\qvar$ be a quasivariety and $\Alg{A} \in \qvar$ be a  non-trivial finite algebra. $\Alg{A}$ is said to be $\qvar$-irreducible, if $\Alg{A}$ is not (isomorphic to) a subdirect product of algebras from $\qvar$ having less elements then $\Alg{A}$. And $\Alg{A}$ is said to be \textit{weakly} $\qvar$-\textit{projective}, if $\Alg{A}$ embeds in every its homomorphic preimage from $\qvar$ (comp. \cite{GorbunovBookE}). And we say that an algebra $\Alg{A}$ is \textit{totally non-projective}, if $\Alg{A}$ is not weakly projective in the quasivariety $\Qvar{\Alg{A}}$ it generates. It is easy to see that a totally non-projective algebra is not weekly $\qvar$-projective in any quasivariety $\qvar$ it belongs to.

\begin{example} Algebra $\Alg{C}_7'$ corresponding to frame $\mathcal{C}_7'$ depicted at Fig.\ref{fig-totnonpr} is totally non-projective. Indeed, algebra $\Alg{C}_5'$ is a subalgebra of $\Alg{C}_7'$ and, therefore, $\Alg{C}_5' \in \Qvar{\Alg{C}_7'}$. Also, $\Alg{C}_{10}'$ is a subdirect product of $\Alg{C}_5'$ and $\Alg{C}_7'$, so, $\Alg{C}_{10}' \in \Qvar{\Alg{C}_7'}$. But$\Alg{C}_7'$ is a homomorphic image of $\Alg{C}_{10}'$, and $\Alg{C}_7'$ is not embeddable in $\Alg{C}_{10}'$. Hence, $\Alg{C}_7'$ is not weakly $\Qvar{\Alg{C}_7'}$-projective, that is, $\Alg{C}_7'$ is totally non-projective.

\begin{figure}[ht]
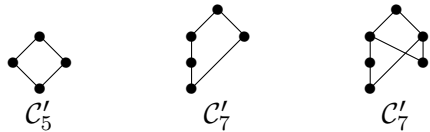

\[
\begin{array}{ccccc}
\ctdiagram{
\ctnohead
\ctinnermid
\ctel 0,0,10,10:{}
\ctel 0,0,-10,10:{}
\ctel 0,20,10,10:{}
\ctel 0,20,-10,10:{}
\ctv 0,0:{\bullet}
\ctv 10,10:{\bullet}
\ctv -10,10:{\bullet}
\ctv 0,20:{\bullet}
}
& \quad \quad &
\ctdiagram{
\ctnohead
\ctinnermid
\ctel 0,0,0,20:{}
\ctel 0,0,20,20:{}
\ctel 0,20,10,30:{}
\ctel 20,20,10,30:{}
\ctv 0,0:{\bullet}
\ctv 0,10:{\bullet}
\ctv 0,20:{\bullet}
\ctv 10,30:{\bullet}
\ctv 20,20:{\bullet}
}
& \quad \quad &
\ctdiagram{
\ctnohead
\ctinnermid
\ctel 0,0,0,20:{}
\ctel 20,10,20,20:{}
\ctel 0,20,10,30:{}
\ctel 20,20,10,30:{}
\ctel 0,20,20,10:{}
\ctel 0,0,20,20:{}
\ctv 0,0:{\bullet}
\ctv 0,10:{\bullet}
\ctv 0,20:{\bullet}
\ctv 10,30:{\bullet}
\ctv 20,20:{\bullet}
\ctv 20,10:{\bullet}
}
\\
\mathcal{C}_5'&&\mathcal{C}_{7}'&& \mathcal{C}_{7}'    
\end{array}
\]
\caption{Example of totally non-projective algebra.} \label{fig-totnonpr}
\end{figure}
\end{example}

Let us note the following simple but nevertheless helpful proposition.   

\begin{prop} Any quasivariety containing a totally non-projective algebra is not primitive.
\end{prop}

\begin{cor} \label{cor-oddcycl} A quasivariety generated by a cyclic Heyting algebra $\Alg{C}_{2m+1}$ of cardinality $2m+1$ is not primitive for any $m \ge 5$. 
\end{cor}
\begin{proof} The proof immediately follows from the observation that for every $m \ge 5$ cyclic algebra $\Alg{C}_{2m+5}$ contains a subalgebra (isomorphic to) $\Alg{C}_7'$, which is totally non-projective.
\end{proof}

Let $\var{Q}$ be a quasivariety. Recall that $\var{Q}$ is called \textit{locally finite} if every finitely-generated algebra from $\var{Q}$ is finite. 

We will use the following criterion of primitiveness of locally finite quasivarieties from  \cite[Proposition 5.1.24]{GorbunovBookE}.

\begin{prop} \label{pr-prim} A locally finite quasivariety $\var{Q}$ is primitive if and only if every finite subdirectly $\var{Q}$-irreducible algebra is weakly projective in $\var{Q}$.
\end{prop}

Thus, due to Propositions \ref{pr_HScplPrim} and \ref{pr-prim}, in order to prove that a deductive system is not hereditarily structurally complete, it is enough to show that $\var{Q}(\DS)$ contains a $\qvar$-irreducible algebra that is not weakly $\qvar$-projective, provided that $\var{Q}(\DS)$ is locally finite. In the following sections we use this approach to establish that there are continuum many deductive system that are not hereditarily structurally complete.


\textbf{Quasivarieties generated by finite cyclic algebras.}
First, let us consider the infinite cyclic Heyting algebra $\Alg{RN}$ - the Rigier-Nishimura ladder. Let us observe that $\Alg{RN}$ is a subalgebra of $\FAlg{\Heyt}{\omega}$. Since $\Heyt$, regarded as quasivariety, corresponds to $\IPC$ and we know that $\SC{\IPC}$ is $\HSCpl$, we can conclude that $\Qvar{\FAlg{\Heyt}{\omega}}$ is a primitive quasivariety, hence, its subquasivariety generated by $\Alg{RN}$ is also primitive. In other words, 

\begin{prop} \label{pr-RN} The structural completion of the logic of $\Alg{RN}$ is $\HSCpl$.
\end{prop}

Before we turn to the quasivarieties generated by finite cyclic algebras, let us prove the following simple proposition which will be instrumental in what follows.  

\begin{prop} (comp.\cite[Lemma 4.1.10]{Rybakov_Book}) \label{pr-fingen} Let $\Alg{A}$ be an $n$-generated algebra and $\vvar$ be a variety generated by $\Alg{A}$. Then,
\[
\Qvar{\FAlg{\vvar}{\omega}} = \Qvar{\FAlg{\vvar}{n}}.
\] 
\end{prop}
\begin{proof} Due to $\FAlg{\vvar}{n}$ is a subalgebra of $\FAlg{\vvar}{\omega}$, we have $\Qvar{\FAlg{\vvar}{n}} \subseteq \Qvar{\FAlg{\vvar}{\omega}}$. Also, due to every quasivariety contains free algebras, $\Qvar{\FAlg{\vvar}{\omega}}$ is the least (relative $\subseteq$) subquasivariety of $\vvar$ that generates $\vvar$. Hence, we only need to verify that $\FAlg{\vvar}{n}$ generates $\vvar$. Since $\FAlg{\vvar}{n} \in \vvar$, all identities valid in $\vvar$ are valid in $\FAlg{\vvar}{n}$, and we need to show that if an identity $\tau$ is refuted in $\vvar$, then $\tau$ is refuted in $\FAlg{\vvar}{n}$ too.

Indeed, suppose $\tau$ is an identity refuted in $\vvar$. Then, due to $\Alg{A}$ generates $\vvar$, this identity is refuted in $\Alg{A}$. Recall that $\Alg{A}$ is $n$-generated and, hence, $\Alg{A}$ is a homomorphic image of $\FAlg{\vvar}{n}$. Hence, $\tau$ cannot be valid in $\FAlg{\vvar}{n}$. 
\end{proof}

Now, let us turn to the quasivarieties generated by finite cyclic algebras. First, we recall that if $\Alg{A}$ is a finite algebra, then quasivariety $\Qvar{\Alg{A}}$ is locally finite, and any  non-trivial finite $\Qvar{\Alg{A}}$-irreducible algebra is embedded in $\Alg{A}$ (see e.g. \cite[Proposition 3.1.6]{GorbunovBookE}). Secondly, by Proposition \ref{pr-fingen}, if $\Alg{A}$ is cyclic, that is, $\Alg{A}$ is generated by a single element, 
\[
\Qvar{\FAlg{\Qvar{\Alg{A}}}{\omega}} = \Qvar{\FAlg{\Qvar{\Alg{A}}}{1}}. 
\]
Thus, if $\Alg{A}$ is cyclic, in order to establish that $\Qvar{\FAlg{\Qvar{\Alg{A}}}{\omega}}$ is primitive it is necessary and sufficient to verify that every $\Qvar{\FAlg{\Qvar{\Alg{A}}}{1}}$-irreducible subalgebra of $\FAlg{\Qvar{\Alg{A}}}{1}$ is weakly $\Qvar{\FAlg{\Qvar{\Alg{A}}}{1}}$-projective.

Let $\Alg{C}_n, n >1 $ denotes a cyclic Heyting algebra having $n$ elements. Then the following holds.

\begin{theorem} \label{th-rnpr} Quasivariety $\Qvar{\Alg{C}_n}$ is primitive if and only if $n = 2,3,4,5,$ $6,8,9,10,12,14$.
\end{theorem}
\begin{proof} The primitiveness of $\Qvar{\Alg{C}_n}$ for $n = 2,3,4,5,6,8,9$ follows immediately from the criterion from \cite{Citkin_1978}. 
By Corollary \ref{cor-oddcycl}, $\Qvar{C_{2k+1}}$ for all $k \ge 5$ are not primitive. 

Let us prove that $\Qvar{C_{2k}}$ for all $k \ge 8$ are not primitive, and this will leave us only with cases $n =10,12,14$.

\begin{figure}[ht]
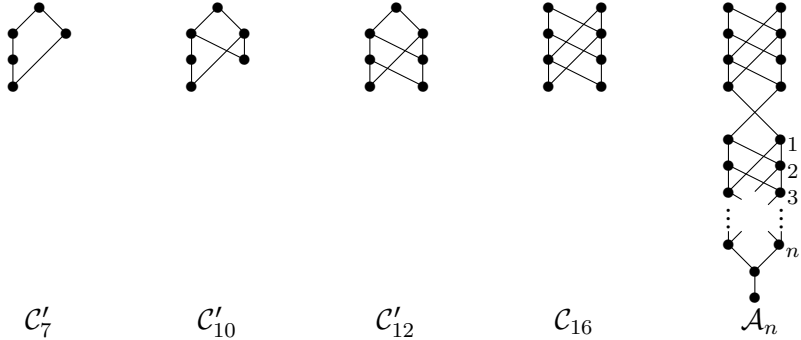

\[
\begin{array}{ccccccccccc}

\ctdiagram{
\ctnohead
\ctinnermid
\ctel 0,0,0,20:{}
\ctel 0,0,20,20:{}
\ctel 0,20,10,30:{}
\ctel 20,20,10,30:{}
\ctv 0,0:{\bullet}
\ctv 0,10:{\bullet}
\ctv 0,20:{\bullet}
\ctv 10,30:{\bullet}
\ctv 20,20:{\bullet}
}

& \quad \quad &
\ctdiagram{
\ctnohead
\ctinnermid
\ctel 0,0,0,20:{}
\ctel 20,10,20,20:{}
\ctel 0,20,10,30:{}
\ctel 20,20,10,30:{}
\ctel 0,20,20,10:{}
\ctel 0,0,20,20:{}
\ctv 0,0:{\bullet}
\ctv 0,10:{\bullet}
\ctv 0,20:{\bullet}
\ctv 10,30:{\bullet}
\ctv 20,20:{\bullet}
\ctv 20,10:{\bullet}
}

& \quad \quad &

\ctdiagram{
\ctnohead
\ctinnermid
\ctel 0,0,0,20:{}
\ctel 20,0,20,20:{}
\ctel 0,20,10,30:{}
\ctel 20,20,10,30:{}
\ctel 0,20,20,10:{}
\ctel 0,0,20,20:{}
\ctel 0,10,20,0:{}
\ctv 0,0:{\bullet}
\ctv 0,10:{\bullet}
\ctv 0,20:{\bullet}
\ctv 10,30:{\bullet}
\ctv 20,20:{\bullet}
\ctv 20,10:{\bullet}
\ctv 20,0:{\bullet}
}

& \quad \quad &


\ctdiagram{
\ctnohead
\ctinnermid
\ctel 0,0,0,30:{}
\ctel 20,0,20,30:{}
\ctel 0,20,20,10:{}
\ctel 0,0,20,20:{}
\ctel 0,10,20,0:{}
\ctel 0,30,20,20:{}
\ctel 0,10,20,30:{}
\ctv 0,0:{\bullet}
\ctv 0,10:{\bullet}
\ctv 0,20:{\bullet}
\ctv 0,30:{\bullet}
\ctv 20,30:{\bullet}
\ctv 20,20:{\bullet}
\ctv 20,10:{\bullet}
\ctv 20,0:{\bullet}
}
& \quad \quad &
\ctdiagram{
\ctnohead
\ctinnermid
\ctel 0,0,0,30:{}
\ctel 20,0,20,30:{}
\ctel 0,20,20,10:{}
\ctel 0,0,20,20:{}
\ctel 0,10,20,0:{}
\ctel 0,30,20,20:{}
\ctel 0,10,20,30:{}
\ctel 0,0,20,-20:{}
\ctel 20,0,0,-20:{}
\ctel 20,-40,20,-20:{}
\ctel 0,-40,0,-20:{}
\ctel 0,-20,20,-30:{}
\ctel 0,-30,20,-40:{}
\ctel 0,-40,5,-43:{}
\ctel 20,-20,0,-40:{}
\ctel 20,-30,10,-40:{}
\ctel 20,-40,15,-45:{}
\ctel 20,-60,10,-70:{}
\ctel 0,-60,10,-70:{}
\ctel 10,-80,10,-70:{}
\ctel 0,-60,0,-55:{}
\ctel 0,-60,5,-55:{}
\ctel 20,-60,20,-55:{}
\ctel 20,-60,15,-55:{}
\ctv 0,0:{\bullet}
\ctv 0,10:{\bullet}
\ctv 0,20:{\bullet}
\ctv 0,30:{\bullet}
\ctv 20,30:{\bullet}
\ctv 20,20:{\bullet}
\ctv 20,10:{\bullet}
\ctv 20,0:{\bullet}
\ctv 0,-20:{\bullet}
\ctv 22,-20:{\bullet_1}
\ctv 0,-30:{\bullet}
\ctv 22,-30:{\bullet_2}
\ctv 0,-40:{\bullet}
\ctv 22,-40:{\bullet_3}
\ctv 0,-50:{\vdots}
\ctv 20,-50:{\vdots}
\ctv 0,-60:{\bullet}
\ctv 22,-60:{\bullet_n}
\ctv 10,-70:{\bullet}
\ctv 10,-80:{\bullet}
}
\\
\mathcal{C}_7' && \mathcal{C}_{10}' && \mathcal{C}_{12}' && \mathcal{C}_{16} && \mathcal{A}_n 
\end{array}
\]
\caption{Frames for proof of non-primitiveness.} \label{fig_Q16}
\end{figure}

Indeed, consider algebras $\Alg{C}_7',\Alg{C}_{10}',\Alg{C}_{12}',\Alg{C}_{16}$ corresponding to frames depicted at Fig.\ref{fig_Q16} and let $\qvar_{16} \bydef \Qvar{\Alg{C}_{16}}$. Algebras $\Alg{C}_{10}',\Alg{C}_{12}'$ are subalgebras of $\Alg{C}_{16}$ and, hence, $\Alg{C}_{10}',\Alg{C}_{12}' \in \qvar_{16}$. Algebra $\Alg{C}_{10}'$ is $\qvar_{16}$-irreducible, for algebra $\Alg{C}_7' \notin \qvar_{16}$ and it is a subdirect factor of $\Alg{C}_{10}'$. And, $\Alg{C}_{10}'$ is a homomorphic image of $\Alg{C}_{12}'$, but it is not a subalgebra of $\Alg{C}_{12}'$. Hence, $\Alg{C}_{10}'$ is not weakly $\qvar_{16}$-projective and quasivariety $\qvar_{16}$ is not primitive.  

The above proof holds for any algebra $\Alg{C}_{2k}$ for any $k \ge 8$.

The cases $n =10,12,14$ can be checked by listing all the $\qvar$-irreducible subalgebras of respective algebras and verifying that all of them are weakly $\qvar$-projective.
\end{proof}

\begin{cor} \label{cor-vnonpr} If $\vvar$ is a variety such that cardinality of $\FAlg{\vvar}{1}$ is finite and exceeds 13, then $\Qvar{\FAlg{\vvar}{\omega}}$ is not primitive. 
\end{cor}
\begin{proof} Let $\vvar$ be a variety and $\FAlg{\vvar}{1}$ has $n$ elements and $n >13$. Then, by Theorem \ref{th-rnpr}, $\Qvar{\FAlg{\vvar}{1}}$ is not primitive. So, $\Qvar{\FAlg{\vvar}{\omega}}$ is not primitive, because $\Qvar{\FAlg{\vvar}{1}}$ is a subquasivariety of $\Qvar{\FAlg{\vvar}{\omega}}$. 
\end{proof}

In other words, the following holds.
\begin{cor} \label{cor-lnonpr} If $\LogL$ is a logic whose Lindenbaum algebra of formulas on one variable is finite and has at least 14 elements, then structural completion of $\LogL$ is not $\HSCpl$. In particular, structural completions of deductive systems $\ds{\BDn}{\MP}$ for all $ n > 3$ are not $\HSCpl$.
\end{cor}
\begin{proof} The proof follows from the known fact that for $n > 3$ cyclic Lindenbaum algebra of $\BDn$ is finite and has more then 14 elements.
\end{proof}

Now, we can prove the main theorem of this section.

\begin{theorem} There is continuum many superintuitionistic logic structurally completion of which is not $\HSCpl$. 
\end{theorem}
\begin{proof} We will use Corollaries \ref{cor-lnonpr} and \ref{cor-vnonpr} and construct continuum many varieties free cyclic algebras of which contain more then 13 elements. We will use the Jankov's argument: let $N_2$ be a set of all natural numbers greater than 2 and let us consider algebras $\Alg{A}_m, m \in N_2$ corresponding to frames $\mathcal{A}_m$ depicted at Fig.\ref{fig_Q16}. Let $I$ be an arbitrary set of $N_2$, and let $\vvar_I$ be a variety generated by algebras $\Alg{A}_m, m \in I$.  Observe that $\Alg{C}_{16}$ is a homomorphic image of algebras $\Alg{A}_m$, hence $\FAlg{\vvar_I}{1}$ has at least 16 elements. On the other hand, algebra $\Alg{C}_{19}$ - the cyclic algebra with 19 elements - is not embedded in either of algebras $\Alg{A}_m, m \in I$ or their homomorphic images. Hence, the characteristic formula $X(\Alg{C}_{19})$ (see \cite{Jankov_1969}) of $\Alg{C}_{19}$ is valid in each $\Alg{A}_m, m \in I$ and, hence, $X(\Alg{C}_{19})$ is valid in $\vvar_I$. Thus, $\Alg{C}_{19} \notin \vvar_I$, and this means that $\FAlg{\vvar_I}{1}$ is finite, because $\Alg{C}_{19}$ is a homomorphic image of $\Alg{RN}$, i.e. $\Alg{RN} \notin \vvar_I$. So, we have established that $\FAlg{\vvar_I}{1}$ is finite and has at least 16 elements. It is clear that there is continuum many subsets of $N_2$ and all we need is to prove that $I$ uniquely defines $\vvar_I$, that is, we need to prove that if $I_1,I_2 \subseteq N_2$ and $I_1 \neq I_2$, then $\vvar_{I_1} \neq \vvar_{I_2}$. But the latter follows from the properties of characteristic formulas (see \cite{Jankov_1969})) and the observation that if $n \neq m$, algebra $\Alg{A}_n$ is not embedded in any homomorphic image of $\Alg{A}_m$. 
\end{proof}

\subsection{Absence of the Least $\HSCpl$ Deductive System}

The goal of this Section is to demonstrate that there is not the least hereditarily complete deductive system extending $\IPC$ and, hence, the criterion similar to the one from \cite{Citkin_1978}, is impossible. 

\begin{theorem} \label{th-noleast} There is no least $\HSCpl$ deductive system above $\IPC$.
\end{theorem}
\begin{proof} First, observe that $\SC{\IPC}$ is a minimal (relative to $\DedExt$) hereditarily structurally deductive system extending $\IPC$. Indeed, all systems between $\IPC$ and $\SC{\IPC}$ have the same logic, namely $\Int$. Thus, except for $\SC{\IPC}$, all these systems are not even structurally complete. So, it is enough to present a $\HSCpl$ deductive system $\DS$ such that $\DS$ that is not an extension of $\SC{\IPC}$.

Let $\LogL_7$ be a set of all formulas valid in $\Alg{C}_7$ - cyclic Heyting algebra with 7 elements and ket $\DS_7 \bydef \ds{\LogL_7}{\MP}$. We will prove that the following holds

\begin{tabular}{rl}
(a) & $\SC{\DS}_7$ is hereditarily structurally complete;\\
(b) & $\SC{\IPC} \ \not\DedExt \ \SC{\DS}_7$.
\end{tabular}

\textbf{Proof of (a)}. First, note that $\SC{\DS}_7$, as any structural completion, is trivially structurally complete, and we only need to demonstrate that all its proper extensions are structurally complete. But every proper extension of $\SC{\DS}_7$ has a logic that is a proper extension of $\LogL_7$, and in \cite{Citkin_1978} it had been proven that all such logics are even hereditarily structurally complete. 

\textbf{Proof of (b)}. Recall from \cite{Citkin1977} that the following substitution instance of Visser's rule $V_1$ (also known as generalized Mints' rule) is admissible in $\IPC$:
\[
M \bydef r \lor ((p_1 \to q) \to (p_1 \lor p_2))/r \lor ((p_1 \to q) \to p_1) \lor ((p_1 \to q) \to p_2).
\]
Therefore, $M$ is derivable in $\SC{\IPC}$ and all its extensions. At the same time, this rule is not admissible in $\LogL_7$: take
\[
p_1 = \neg\neg q, \ p_2 = \neg q, \text{ and } r = (\neg\neg q \to q).
\]
On one hand, we have
\[
(\neg\neg q \to q) \lor ((\neg\neg q \to q) \to(\neg\neg q \lor \neg q)) \in \LogL_7.
\]
On the other hand, we have
\[
(\neg\neg q \to q) \lor ((\neg\neg q \to q) \to \neg\neg q) \lor ((\neg\neg q \to q) \to \neg q) \notin \LogL_7, 
\]
because by the Glivenko Theorem, the above formula is equal in $\IPC$ to the following formula
\[
(\neg\neg q \to q) \lor \neg\neg q \lor \neg q,
\]
and the latter formula is not valid in $\Alg{C}_7$, that is, it is not a theorem of $\LogL_7$. So, we have established that rule $M$ is derivable in $\SC{IPC}$ and is not derivable in $\SC{\LogL_7}$, which proves (b). 
\end{proof}

\subsection{Hereditary Structural Incompleteness of $\SC{\KP}$ and $\SC{\ML}$}

Recall that $\KP$ denotes Kreisel-Putnam's logic and $\ML$ denotes Medvedev's logic. In this Section we to prove that $\SC{\KP}$ is not hereditarily structurally complete (structural incompleteness of $\KP$ was observed in \cite{Wojtylak_Problem_2004}). The author is grateful to E.~Je{\v{r}}{\'a}bek who suggested the idea of the proof. 

First, we recall that $\ML$ is structurally complete, due to \cite{Prucnal_Structural_1976}, $\ML$ is structurally complete, that is, every admissible in $\ML$ rule is derivable in it. Hence,
\begin{equation}
\SC{\ML} \DedEqv \ML.  \label{eq-MLScpl} 
\end{equation}

In \cite{Levin_Some_1969} Levin had constructed a class $\Frm^\circ$ of formulas\footnote{The definition of this class is irrelevant for our purposes, but the reader can find it in \cite{Levin_Some_1969}.} that posses the following property.

\begin{prop} (see \cite[Theorem 4]{Levin_Some_1969}) \label{pr-Levin} For any formula $A$, 
\[
\vdash_\ML A  \text{ if and only if  }  \vdash_\KP \sigma(A) \text{ for every substitution } \sigma: \Vars \to \Frm^\circ. \label{eq-pr-Levin}
\]
\end{prop}
So, in a way, $\ML$ is reduced to $\KP$. Let us consider how such reduction is linked to admissibility.


\textbf{Reducibility of Deductive Systems.}
Let $\DS_1$ and $\DS_2$ be deductive systems and $\Sigma$ be a set of substitutions. Then $\DS_1$ is $\Sigma$\textit{-reducible} to $\DS_2$ if for any formula  $A$,
\begin{equation}
\vdash_{\DS_1} A \text{ if and only if for every } \sigma \in \Sigma, \ \vdash_{\DS_2} \sigma(A). \label{sigmared}
\end{equation}

\begin{prop} \label{th-sigmatred} Let $\DS_1$ and $\DS_2$ be deductive systems, $\Sigma$ be a set of substitutions, and $\DS_1$ be $\Sigma$-reducible to $\DS_2$. Then every admissible in $\DS_2$ rule is admissible in $\DS_1$.
\end{prop}
\begin{proof} Suppose a rule $\ruleR$ is admissible in $\DS_2$. We need to prove that $\ruleR$ is admissible in $\DS_1$. For this we prove the inverse statement: if $\ruleR$ is not admissible in $\DS_1$, then $\ruleR$ is not admissible in $\DS_2$. 

Suppose rule $\ruleR \bydef A_1,\dots,A_n/B$ is not admissible in $\DS_1$. Then there is a substitution $\sigma$ such that
\[
\vdash_{\DS_1} \sigma(A_i) \text{ for all } i=1,\dots,n  \text{, while}\nvdash_{\DS_1} \sigma(B).
\]
Since $\nvdash_{\DS_1}  B$, $\Sigma$-reducibility entails that there is a substitution $\sigma' \in \Sigma$ such that 
\begin{equation}
\nvdash_{\DS_2} \sigma'(\sigma(B)). \label{pr-sigmatred2}
\end{equation}
On the other hand, $\Sigma$-reducibility entails that for every $i=1,\dots,n$
\begin{equation}
\vdash_{\DS_2} \sigma'(\sigma(A_i)). \label{pr-sigmatred3}
\end{equation}
And \eqref{pr-sigmatred3} and \eqref{pr-sigmatred2} mean that $\ruleR$ is not admissible in $\DS_2$.
\end{proof}

\begin{cor} \label{cor-reduc} If a deductive system $\DS_1$ is $\Sigma$-reducible to $\DS_2$, then 
$\SC{\DS}_2 \ \DedExt \ \SC{\DS}_1$. Hence, if $\SC{\DS}_2$ is hereditarily structurally complete, so is $\SC{\DS}_1$.
\end{cor}

\textbf{The Case of $\ML$.}
Using Levin's Theorem and Corollary \ref{cor-reduc} we can prove the following theorem. 

\begin{theorem} \label{th-KPML}Every admissible in $\KP$  rule is derivable in $\ML$. 
\end{theorem}
\begin{proof} Indeed, \eqref{eq-pr-Levin} means that $\ML$ is $\Sigma'$-reducible to $\KP$, where $\Sigma' \bydef \set{\sigma \in \Sigma}{\sigma: \Vars \to \Frm^\circ}$. Hence, by Corollary \ref{cor-reduc},
\begin{equation}
\SC{\KP} \DedExt \SC{\ML}. \label{eq-th-KPML-1}
\end{equation}
Now, we can use \eqref{eq-MLScpl} and obtain
\begin{equation}
\SC{\KP} \DedExt \SC{\ML} \DedEqv \ML.  \label{eq-th-KPML-2} 
\end{equation}
And \eqref{eq-th-KPML-2} means that every admissible in $\KP$ rule is derivable in $\ML$.
\end{proof}

Recall also (see \cite{Citkin_1978}), that $\ML$ is not hereditarily structurally complete. Hence, by Corollary \ref{cor-reduc} and \eqref{eq-th-KPML-2}, we have

\begin{cor} $\SC{\KP}$ is not hereditarily structurally complete.
\end{cor}

\subsection{Hereditary Structural Completeness and Finite Model Property}

In this Section we show that there is continuum many hereditarily structurally complete deductive systems that cannot be defined by their finite models, which is different from the situation with deductive systems with $\MP$ as a single inference rule. 

A logic $\LogL$ is said to have the \textit{finite model property} (fmp for short) if for every formula $A \notin \LogL$ there is a finite model of $\LogL$ in which $A$ is refuted. A logic $\LogL$ has the \textit{finite model property relative to admissibility} (a-fmp for short), if for every rule $\ruleR$ not admissible in $\LogL$ there is a finite model of $\LogL$ in which all admissible in $\LogL$ rules are valid and $\ruleR$ is not valid, that is, there is a finite model of $\SC{\ds{\LogL}{\MP}}$ that refutes $\ruleR$. In other words, $\DS$ has the a-fmp if each rule valid in every finite model of $\SC{\DS}$ is admissible in $\DS$. 

In \cite{Rybakov_Intermediate_1993} Rybakov described all superintutionistic logics enjoying a-fmp. In \cite[Section 3.6]{Goudsmit_PhD} Goudsmit presented some classes of superintuitionistic logics that do not have the a-fmp. The normal modal logics with and without the a-fmp are studied in \cite{Rybakov_Kiatkin_Fmp_1999, Rybakov_Kiatkin_2001}. 

In this Section we establish connections between a-fmp and hereditary structural completeness.

\begin{theorem} \label{th-a-fmp} If a deductive system $\DS$ admits $V_1$ and enjoys the a-fmp, then $\SC{\DS}$ is hereditarily structural complete.
\end{theorem}
\begin{proof} First, recall that the generalized Mints rule $M$ is a substitution instance of $V_1$, hence, admissibility of $V_1$ entails admissibility of $M$. Now, we can apply \cite[Corollary 2]{Citkin_1977_D1} and conclude that all admissible in $\Int$ rules, and, therefore, all Visser's rules are valid in every finite model of $\SCDS$. Since $\DS$ enjoys the a-fmp, all Visser's rules are admissible in $\SC{\DS}$, and this means that $\SCDS$ is a deductive extension of $\SC{\IPC}$. By Theorem \ref{th-SCIPC}, $\SC{\IPC}$ is hereditarily structurally complete, so, $\SCDS$ is hereditarily structurally complete.
\end{proof}

Let us note that a-fmp is a property of structural completion of a deductive system $\DS$ rather than property of $\DS$ per se. This is to say that, any two logically equivalent deductive systems either both have the a-fmp, or both do not have the a-fmp. Hence, if a deductive system $\DS \bydef \ds{\Ax}{\Rules}$ enjoys the a-fmp, then deductive system $\ds{\LogL(\DS)}{\MP}$ enjoys the a-fmp. The Theorem 7 of \cite{Rybakov_Intermediate_1993} (see also \cite[Theorem 6.3.5]{Rybakov_Book}) states that there is continuum many deductive extensions of $\SC{\IPC}$. On the other hand, there is only countable many hereditarily structural complete superintuitionistic logics (see e.g. \cite[Theorem 5.4.10]{Rybakov_Book}). Hence, the following holds.

\begin{cor} \label{cor-a-fmp-card}  There is continuum many hereditarily structural complete deductive systems without the a-fmp.
\end{cor}

\subsection{Open Problems}

In conclusion, let us point out some open problems. We start with decidability of hereditary structural completeness that may be presented in two different ways.

\begin{problem} \label{prb1} Is there an algorithm that, given a formula $A$, decides whether the structural completion of $\ds{\Ax^i + A}{\MP}$ is hereditarily structural complete? 
\end{problem}

\begin{problem} \label{prb2} Is there an algorithm that, given a finite Heyting algebra $\Alg{A}$, decides whether the structural completion of logic of $\Alg{A}$ (that is, the structural completion of  $\ds{\LogL(\Alg{A})}{\MP}$) is hereditarily structural complete? 
\end{problem}

Since structural completion of $\IPC$ is $\HSCpl$, the following problem is important.

\begin{problem} \label{prob-prim1} Is there an algorithm that, given a formula $A$, decides whether the logic $\ds{\Ax^i + A}{\MP}$ admits all admissible in $\IPC$ rules (i.e. admits all Visser's rules)? 
\end{problem}

Let us note that, due to Theorem 6.3.6 from \cite{Rybakov_Book}, there is an algorithm that, given a finite Heyting algebra $\Alg{A}$, decides whether logic $\LogL(\Alg{A})$ admits all Visser's rules.

\begin{problem} \label{prob-prim2} Is there an algorithm that, given a finite Heyting algebra $\Alg{A}$, decides whether a deductive system defined by $\Alg{A}$ is hereditarily structural complete? In other words, is there an algorithm that, given a finite Heyting algebra $\Alg{A}$, decides whether quasivariety $\Qvar{\Alg{A}}$ is primitive?
\end{problem}

Let us observe that in spite of Proposition \ref{pr-fingen}, the positive answer to Problem \ref{prob-prim2} entails a positive answer to Problem \ref{prb2}. 

\bibliographystyle{acm}

\end{document}